\documentclass[12pt]{article}
\usepackage{fullpage}
\usepackage{times}
\usepackage{amssymb}
\usepackage{amsthm}
\usepackage{graphics}
\usepackage{epsfig}
\usepackage{amssymb}
\usepackage{graphicx}
\usepackage{subfigure}
\usepackage{fullpage}
\usepackage{color}
\usepackage{tikz}
\usepackage{tikz,pgfplots}
\pgfplotsset{compat=1.12,axis lines=center}
\usepgfplotslibrary{fillbetween}
\usetikzlibrary{patterns}
\usepackage{hyperref}
\usepackage{a4wide}
\usepackage{amsmath}
\usepackage{amsfonts}
\usepackage{enumerate}
\usepackage{multicol}

\newcommand{\PP}{ \mathrm{SymPoly}}
\newcommand{\QQ}{ \mathbb{Q}}

\newcommand{\floor}[1]{\left\lfloor #1 \right\rfloor}
\newcommand{\nmrhalf}{\lfloor (n-r)/2 \rfloor}

\newcommand{\Perm}{\mathrm{perm}}
\newcommand{\MP}{\mathrm{MP}}
\newcommand{\MPD}{\mathrm{MPD}}
\newcommand{\mpd}{\mathrm{mpd}}
\newcommand{\SMPDE}{\mathrm{SgnMPDE}}
\newcommand{\MPE}{\mathrm{MPE}}
\newcommand{\SMPE}{\mathrm{SgnMPE}}
\newcommand{\mpa}{\mathrm{mp}}  
\newcommand{\SSS}{\mathfrak{S}}
\newcommand{\AAA}{\mathcal{A}}

\newcommand{\inv}{\mathrm{inv}}
\newcommand{\des}{\mathrm{des}}
\newcommand{\exc}{\mathrm{exc}}
\newcommand{\stat}{\mathrm{stat}}

\newcommand{\Exc}{\mathrm{Exc}}
\newcommand{\sgn}{\mathrm{sgn}}
\newcommand{\SD}{ \mathrm{SgnDes}}

\newcommand{\DSE}{ \mathrm{DSgnExc}}

\newtheorem*{theoremaux}{Theorem \theoremauxnum}
\gdef\theoremauxnum{1}

\newtheorem{lemma}{\bf Lemma}[section]
 
\newtheorem{theorem}{\bf Theorem}[section]
\newtheorem{proposition}[lemma]{\bf Proposition}

\newtheorem{definition}{\bf Definition}[section]
\newtheorem{remark}{\bf Remark}[section]

\title{Sign-balance of excedances over mod-k-alternating permutations and gamma-positivity}

\author{Hiranya Kishore Dey \thanks{Department of Mathematics, Indian Institute of Science, Bangalore 560012, India. Email: hiranya.dey@gmail.com, hiranyadey@iisc.ac.in} \and Iswar Mahato\thanks{Department of Mathematics, Indian Institute of Technology Bombay, Mumbai 400076, India. Email: iswarmahato02@gmail.com, iswar@math.iitb.ac.in}}
\date{\today}
\begin{document}
\maketitle

\begin{abstract}
A permutation is called {\it mod-k-alternating} if its entries are restricted to having the same remainder as the index, modulo some integer $k \geq 1.$ In this paper, we find the sign-balance for 
mod-k-alternating permutations with respect to the statistic excedance. Moreover, we study the sign-balance for excedances over mod-k-alternating derangements. The results are obtained by constructing suitable matrices and connecting their determinants with the signed excedance enumeration of mod-k-alternating permutations. As an application of the signed excedance enumeration, we prove that when $n \equiv k \pmod {2k}$, the excedance enumerating polynomials over the even and odd mod-k-alternating permutations, starting with a fixed remainder, are gamma-positive. 
\end{abstract}


\medskip 


\section{Introduction}\label{sec:intro}

Enumerating permutations by different statistics is an interesting and well-studied area in combinatorics. The most famous statistics on a permutation are descent, excedance, inversion and major-index, and their enumerators are classical results in combinatorics (see \cite{Foata_netto}, \cite{Macmahon-indicesofperm}). One can look at the course notes of Foata and Han \cite{Foata-Han-qseries} and the book of Petersen \cite{Petersen-book} for a self-contained introduction to this area. 

For a positive integer $n$, let $\mathfrak{S}_n$ be the symmetric group on the set $[n]=\{1,2,\hdots,n\}$. 
For a permutation $\pi$, written in one-line notation as $\pi=\pi_1 \pi_2 \dots \pi_n$, let $\mathrm{Des}(\pi)= \{ i \in [n-1]: \pi_i> \pi_{i+1} \}$ be its set of descents and let $\mathrm{des}(\pi)= |\mathrm{Des}(\pi)|.$ The classical Eulerian number, denoted by $A_{n,k}$, is the number of permutations in $\SSS_n$
with $k$ descents and $A_n(t)= \sum _{\pi \in \SSS_n} t^{\des(\pi)} $ is the Eulerian polynomial. Define   $\Exc(\pi)= \{ i \in [n-1]: \pi_i > i \}$ as its set of excedances and let $\mathrm{exc}(\pi)= |\mathrm{Exc}(\pi)|.$ It is well known (see, \cite{Macmahon-indicesofperm}) that excedances and descents are equidistributed when summed over the elements of $\SSS_n,$ that is,
$\sum_{\pi \in \SSS_n}t^{\des(\pi)}= \sum_{\pi \in \SSS_n} t^{\exc(\pi)}.$ 

For $\pi \in \SSS_n,$ the number of inversions of $\pi$ is defined as 
$\inv(\pi)= |\{(i,j):i<j, \pi_i> \pi_j \}| $ and the sign of a permutation is defined as $\sgn(\pi)=(-1)^{\inv(\pi)}.$ It is well-known that the sign of a permutation $\pi \in \SSS_n$ is $1$ if and only if $\pi$ is even, that is, it belongs to the alternating group $\AAA_n$. The sign-balance of a set $T_n \subseteq \SSS_n,$ respecting a statistic $\stat$, is given by 
$$ \sum _{\pi \in T_n} (-1)^{\inv(\pi)} t^{\stat(\pi)}.$$ 

The sign-balance of $\SSS_n$ with respect to the statistic $\des$ was first considered by Loday in \cite{Loday-oper}. He defined $\SD(n,t) = \sum_{\pi \in \SSS_n}(-1)^{\inv(\pi)} t ^{\des(\pi)} $ as the signed descent enumerator and conjectured the following recurrence relation, which was proved by D{\'e}sarm{\'e}nien and Foata in \cite[Theorem 1]{Foata-desar-signed-eul}. 
	
     \begin{theorem}[D{\'e}sarm{\'e}nien-Foata]
	\label{thm:Foata_desar_signed_des}
	For positive integers $n$, we have 
	 \begin{eqnarray*}
	\SD_n(t) & = & \begin{cases}
	 (1-t)^m A_m(t) & \text {if $n=2m$},\\
	 (1-t)^m A_{m+1}(t) & \text {if $n=2m+1$.}
	 \end{cases}
	 \end{eqnarray*} 
	\end{theorem}
 
Later, Wachs in \cite{Wachs-descent} gave a sign-reversing involution on $\SSS_n$, thereby giving an alternate, bijective proof of Theorem \ref{thm:Foata_desar_signed_des}. Gessel and Simion in \cite[Corollary 2]{Wachs-descent} gave an elegant factorial-type product formula for the refined sign-balance of $\SSS_n$ respecting major indices. 

There have been many works on the sign-balance of some restricted subsets of the symmetric group $\SSS_n$. For example, Simion and Schmidt \cite{Simion-Schmidt} determined the sign-balance for the 321-avoiding permutations in $\SSS_n$ with respect to various statistics. Adin and Roichman \cite{Adin-Roichman321} gave a refinement, respecting the position of the last descent. Reifegerste  \cite{Reifegerste-321-avoiding-sign} proved another refinement, respecting the length of the longest increasing subsequence. The sign-balance of $\SSS_n$ and its various subsets, respecting the statistic $\exc$ have been investigated by Mantaci in \cite{Mantaci-thesis}. He proved the following remarkable result on signed excedance enumeration.  



\begin{theorem}[Mantaci]
	\label{thm:Mant_signed_exc}
	For positive integers $n$, we have 
	 \begin{equation*}
	 \sum_{\pi \in \SSS_n}(-1)^{\inv(\pi)} t ^{\exc(\pi)} =
	 (1-t)^{n-1} . 
	\end{equation*}
	\end{theorem}

Later, Mantaci in \cite{Mantaci-binomial} gave a combinatorial proof of this result. In \cite{Sivasubramanian-adv-exc}, Sivasubramanian gave an alternate proof of Mantaci’s result by evaluating the determinant of an appropriately defined $n \times n$ matrix. Motivated by these works, in this article, we study the sign-balance of the statistic $\exc$ over mod-k-alternating permutations. The results of this article can be considered as an application of linear algebraic techniques in the area of combinatorics, which are obtained by constructing suitable matrices and connecting their determinants with the signed excedance enumeration of mod-k-alternating permutations.



A permutation is called a {\it parity alternating permutation}, if its entries assume odd and even integers alternately. In 2010, Tanimoto \cite{Tanimoto-parityalt-annals} introduced the notion of parity alternating permutation and studied its various combinatorial properties. In \cite{Tanimoto-parityalt-advances}, the author determined the sign-balance of parity alternating permutation, respecting the statistic descent, and described the signed Eulerian numbers by the parity alternating permutations. Recently, Kebede and Rakotondrajao \cite{kebede-rakotondrajao} studied the parity alternating permutations starting with an odd integer. Very recently,  Alexandersson et. al. \cite{alexandersson-rebede} generalized the notion of parity alternating permutations to {\it mod-k-alternating-permutations}, which is defined as follows:

\begin{definition}
For any positive integer $k$, a mod-k-alternating permutation of size $n$ is a permutation $\pi \in \SSS_n$ such that $\pi_i \equiv i \pmod k$ for all $i=1,2,\dots,n.$
\end{definition}

 Let $\MP_n^k$ be the set of such mod-k-alternating permutations in $\SSS_n$ and let $\mpa_n^k$ denote the cardinality of the set $\MP_n^k$. It is easy to see that $\MP_n^k$  is a subgroup of $\SSS_n$. Note that 
 $$\mpa_n^k=\Big(\Big \lceil \frac{n}{k}\Big \rceil !\Big)^{j}\Big(\Big \lfloor \frac{n}{k}\Big \rfloor !\Big)^{k-j},$$
where $j$ is the remainder when $n$ is divided by $k$. For example, $\MP_5^3$ has $(2!)^2\cdot (1!)^1=4$ elements such as $12345, 15342, 42315$ and $45312$.

 Alexandersson et. al. \cite{alexandersson-rebede} also introduced the notion of {\it mod-k-alternating permutations starting with the remainder $r \pmod k$} as follows:

\begin{definition}
Let $\MP_{n,r}^k$ denote the set of permutations of length $n$ that satisfy $\pi _i - i \equiv r-1 \pmod k$ for $1 \leq i \leq n$. Clearly, when $r=1$, we have $\MP_{n,1}^k= \MP_{n}^k$. Moreover, let $\mpa_{n,r}^k= |\MP_{n,r}^k|.$
\end{definition}

In this article, our aim is to study the signed excedance enumeration over $\MP_{n,r}^k$.  We define the polynomials
\[ \MPE_{n,r}^k(t)= \sum _{\pi \in \MP_{n,r}^k} t^{\exc(\pi)} \hspace{3 mm} \mbox{ and } \hspace{3 mm}
\SMPE_{n,r}^k(t)= \sum _{\pi \in \MP_{n,r}^k} (-1)^{\inv(\pi)} t^{\exc(\pi)},\] and prove the following result, which generalizes Theorem \ref{thm:Mant_signed_exc} proved by Mantaci in \cite{Mantaci-thesis}.

\begin{theorem}
\label{thm:signed excedance-modk}
Let $n,k$ be positive integers. If $n$ is not divisible by $k$, say $n=mk+j$ with $1 \leq j \leq k-1$, we have 
\[\SMPE_n^k(t)= \SMPE_{n,1}^k(t)= (1-t)^{n-k} .  \] 
If $n$ is divisible by $k$, say $n=mk$, we have 
\begin{eqnarray*}
\SMPE_{n,r}^k(t)  & = & \begin{cases}
(1-t)^{n-k} & \text {if $r=1$}, \\
(-1)^{(mr-1)(k-r+1)} t^{k-r+1} (1-t)^{n-k} & \text {if $2 \leq r \leq k$}.
 \end{cases}
\end{eqnarray*} 
\end{theorem} 


For a permutation $\pi = \pi_1 \pi_2 \dots \pi_n \in \SSS_n$, we say that $i \in [n]$ is a fixed point of $\pi$ if $\pi_i=i.$  If $\pi$ has no fixed point, then $\pi$ is called a \textit{derangement} and the set of all derangements in $\SSS_n$ is denoted by $\mathcal{D}_n$.  Recall that for a natural number $n$, its $q$-analogue is defined as 
\[ [n]_q=\frac{1-q^n}{1-q}=1+q+q^2+\hdots+q^{n-1}. \]

In \cite{Mantaci-rako-xc-derange}, Mantaci and Rakotondrajao determined the signed excedance enumerator for derangements and obtained the following result.

\begin{theorem}[Mantaci-Rakotondrajao] 
 For $n\geq 2$, let $\DSE_n(t)=\sum_{\pi \in \mathcal{D}_n}(-1)^{\inv(\pi)} t^{\exc(\pi)}$ be the signed excedance enumerator over derangements. Then  $\DSE_n(t)=(-1)^{n-1}t[n-1]_t$.    
\end{theorem}

Recently, Alexandersson and Getachew \cite{alexandersson-getachew} proved some multivariate generalizations of a formula for enumerating signed excedances in derangements. Kebede and Rakotondrajao \cite{kebede-rakotondrajao} studied the parity alternating derangements which are the derangements that also are parity alternating permutations starting with odd integers. Motivated by these works, in this article, we study the signed excedance enumeration over {\it mod-k-alternating derangements starting with the remainder $r \pmod k$}. 

Let $\MPD_{n,r}^k$ be the set of mod-k-alternating permutations starting with remainder $r \pmod k$, which are also derangements and $\mpd_{n,r}^k= |\MPD_{n,r}^k|$. Define
\[\SMPDE_{n,r}^k(t)= \sum_{\pi \in \MPD_{n,r}^k } (-1)^{\inv(\pi)}t^{\exc(\pi)}.\] 
We prove the following result about the 
signed excedance enumeration over mod-k-alternating derangements. 
 \begin{theorem}
 \label{thm:signed_Exc_derange_pap}
 Let $n,k$ be positive integers. If $n=mk+j$ where $0 \leq j \leq k-1$, we have 
\begin{align*}
\SMPDE_{n,r}^k(t)  
= & \begin{cases}
(-1)^{n} (-t)^k \big([m]_t \big) ^ j \big([m-1]_t \big)^{k-j}  & \text {if $r=1$},\\
 \SMPE_{n,r}^k(t)  & \text {if $r \geq 2$.}
\end{cases}
\end{align*} 
\end{theorem}

As an application of the signed excedance enumeration, we also prove that, for suitable choices of $n$ and $k$, the excedance enumerating polynomials over the even and odd mod-k-alternating permutations starting with a fixed remainder $r \pmod k$, are {\it gamma-positive}. To avoid repetition, we will define gamma-positivity in details in Section \ref{sec:gamma-positivity}. Define
\[ \MPE_{n,r}^{k,+}(t)= \sum_{\pi \in \MP_{n,r}^k \cap \AAA_n} t^{\exc(\pi)}, \hspace{3 mm} \mbox{ and } \hspace{3 mm}
\MPE_{n,r}^{k,-}(t)= \sum_{\pi \in \MP_{n,r}^k \cap (\SSS_n \setminus \AAA_n)} t^{\exc(\pi)}. \] In this context, we prove the following result.

\begin{theorem}
\label{thm:gamma-pos-pap-even-odd-n-2mod4}
For positive integers $n\equiv k \pmod {2k}$ with $n \geq 5k$, the polynomials $\MPE_{n,1}^{k,+}(t)$ and $\MPE_{n,1}^{k,-}(t)$  are gamma-positive with center of symmetry  $(n-k)/2.$ Moreover, for $2 \leq r \leq k$, the polynomials $\MPE_{n,r}^{k,+}(t)$ and $\MPE_{n,r}^{k,-}(t)$  are gamma-positive with center of symmetry  $(n+1-r)/2.$ 
\end{theorem}

The article is organized as follows. In Section $2$, we give the proof of Theorem \ref{thm:signed excedance-modk} by constructing suitable matrices and connecting their determinants with the signed excedance enumeration of mod-k-alternating permutations. We also establish a relationship between the un-signed excedance enumerating polynomials and the Eulerian numbers. We discuss the proof of Theorem \ref{thm:signed_Exc_derange_pap} and Theorem \ref{thm:gamma-pos-pap-even-odd-n-2mod4} in Section $3$ and $4$, respectively. 




\section{Excedances and permanents, signed excedances and determinants}

The main goal of this section is to prove Theorem \ref{thm:signed excedance-modk}. This will be achieved by constructing appropriate matrices and connecting their determinants with the signed excedance enumerating polynomials over $\MP_{n,r}^k$. As a by-product of our method, we will also obtain  an explicit expression for the un-signed excedance enumerating polynomials. For that, first let us recall the definition of the {\it permanent} of a matrix.

\begin{definition}
For a square matrix $M$ of size $n \times n$, the permanent of $M$,  denoted by $\Perm(M)$, is defined as 
\[ \Perm(M)  =  \displaystyle  \sum _{\pi \in \SSS_n}  \prod _{i=1}^n m_{i,\pi_i}. \]    
\end{definition}

Now, we state an important lemma which will be used in the subsequent proofs. 

\begin{lemma}[{\cite[Lemma 9]{alexandersson-rebede}}]
\label{lem:cruicial} 
We have $\mpa_{n,r}^k=0$ unless $r=1$ or $n=mk$ for some positive integer $k$.
\end{lemma}

In the following two propositions, we establish a relation among 
the permanents and the determinants of suitably constructed matrices with the un-signed and signed excedance enumerating polynomials over $\MP_{n,r}^k$, respectively. 

\begin{proposition}
\label{prop:mn1}
Let $n$ be a positive integer.  Consider the $n \times n$ matrix $M_{n,1}= (m^1_{i,j})$ defined by
\begin{eqnarray*}
m^1_{i,j} & = & \begin{cases}
0 & \text {if $|i-j|$ is not divisible by $k$}, \\
1 & \text {if $|i-j|$ is divisible by $k$ and $i \geq j$},\\
 t & \text {if $|i-j|$ is divisible by $k$ and $i < j$}. 
 \end{cases}
\end{eqnarray*} 
Then, we have
\[\Perm(M_{n,1})=\MPE_{n,1}^k(t), \hspace{3 mm}  
\mbox{ and } 
\hspace{3 mm}
\det(M_{n,1})=\SMPE_{n,1}^k(t).\]
\end{proposition}

\begin{proof} 
By the definition, it follows that 
\begin{eqnarray*}
  \Perm(M_{n,1}) & = & \displaystyle  \sum _{\pi \in \SSS_n}  \prod _{i=1}^n m^1_{i,\pi_i}  =  \displaystyle  \sum _{\pi \in \MP_{n,1}^k}  \prod _{i=1}^n m^1_{i,\pi_i} + \displaystyle  \sum _{\pi \in (\SSS_n - \MP_{n,1}^k)}  \prod _{i=1}^n m^1_{i,\pi_i}, 
  \end{eqnarray*}
  and
  \begin{eqnarray*}
  \det(M_{n,1}) & = & \displaystyle  \sum _{\pi \in \SSS_n} (-1)^{\inv(\pi)} \prod _{i=1}^n m^1_{i,\pi_i} \\
  & = & \displaystyle  \sum _{\pi \in \MP_{n,1}^k} (-1)^{\inv(\pi)} \prod _{i=1}^n m^1_{i,\pi_i} + \displaystyle  \sum _{\pi \in (\SSS_n - \MP_{n,1}^k)} (-1)^{\inv(\pi)} \prod _{i=1}^n m^1_{i,\pi_i}. 
  \end{eqnarray*}
For a permutation $\pi \in (\SSS_n- \MP_{n,1}^k)$, it is necessary that $|\pi_i-i|$ is not divisible by $k$ for some $1 \leq i \leq n$ and hence $m^1_{i,\pi_i}=0$. Thus, we have 
\[ \Perm(M_{n,1}) =  \displaystyle  \sum _{\pi \in \MP_{n,1}^k}  \prod _{i=1}^n m^1_{i,\pi_i} \hspace{3 mm} \mbox{ and } \hspace{3 mm} \det(M_{n,1}) =  \displaystyle  \sum _{\pi \in \MP_{n,1}^k} (-1)^{\inv(\pi)} \prod _{i=1}^n m^1_{i,\pi_i}.\] 
Now, for $\pi \in \MP_{n,1}^k$, let $T_{\pi} = \prod _{i=1}^n m^1_{i, \pi_i}$ be the term occuring in the permanent (respectively, determinant) expansion corresponding to $\pi$. Since $m^1_{i,j}=t$ if $i < j$ and $m^1_{i,j}=1$ otherwise, we have $T_{\pi} = t^{\exc(\pi) }$. Hence, 

\[ \Perm(M_{n,1})= \displaystyle  \sum _{\pi \in \MP_{n,1}^k}  \prod _{i=1}^n t^{\exc(\pi)} = \MPE_{n,1}^k(t),\]
and 
\[ \det(M_{n,1}) =  \displaystyle  \sum _{\pi \in \MP_{n,1}^k} (-1)^{\inv(\pi)} \prod _{i=1}^n t^{\exc(\pi)} = \SMPE_{n,1}^k(t) .\]  
\end{proof}

\begin{proposition}
\label{prop:gen-mnr}
Let $n$ be a positive integer divisible by $k$.  Consider the $n \times n$ matrix $M_{n, r}= (m^r_{i, j})$ defined by
\begin{eqnarray*}
m_{i,j}^r & = & \begin{cases}
0 & \text {if \quad $j-i   \not \equiv r-1 \pmod k $}, \\
1 & \text {if \quad $j-i  \equiv r-1 \pmod k$ and $i \geq j$},\\
 t & \text {if \quad $j-i  \equiv r-1 \pmod k$ and $i < j$}. 
 \end{cases}
\end{eqnarray*} 
Then, we have
\[\Perm(M_{n,r})=\MPE_{n,r}^k(t), \hspace{3 mm}  
\mbox{ and } 
\hspace{3 mm}
\det(M_{n,r})=\SMPE_{n,r}^k(t).\]
\end{proposition}

\begin{proof}
The proof follows via the similar steps as the proof of Proposition \ref{prop:mn1}. Yet, we give it for the sake of completeness. Note that
  \begin{eqnarray*}
  \Perm(M_{n,r}) & = & \displaystyle  \sum _{\pi \in \SSS_n}  \prod _{i=1}^n m^r_{i,\pi_i} = \displaystyle  \sum _{\pi \in \MP_{n,r}^k}  \prod _{i=1}^n m^r_{i,\pi_i} + \displaystyle  \sum _{\pi \in (\SSS_n - \MP_{n,r}^k)}  \prod _{i=1}^n m^r_{i,\pi_i},
  \end{eqnarray*}
  and
  \begin{eqnarray*}
  \det(M_{n,r}) & = & \displaystyle  \sum _{\pi \in \SSS_n} (-1)^{\inv(\pi)} \prod _{i=1}^n m^r_{i,\pi_i} \\
  & = & \displaystyle  \sum _{\pi \in \MP_{n,r}^k} (-1)^{\inv(\pi)} \prod _{i=1}^n m^r_{i,\pi_i} + \displaystyle  \sum _{\pi \in (\SSS_n - \MP_{n,r}^k)} (-1)^{\inv(\pi)} \prod _{i=1}^n m^r_{i,\pi_i}. 
  \end{eqnarray*}
 Observe that for a permutation $\pi \in (\SSS_n- \MP_{n,r}^k)$, it is necessary that $\pi_i-i \not \equiv r-1 \pmod k$ for some $1 \leq i \leq n$, and hence $m^r_{i,\pi_i}=0$. Thus, we have
\[ \Perm(M_{n,r}) =  \displaystyle  \sum _{\pi \in \MP_{n,r}^k}  \prod _{i=1}^n m^r_{i,\pi_i} \hspace{3 mm} \mbox{ and } \hspace{3 mm} \det(M_{n,r}) =  \displaystyle  \sum _{\pi \in \MP_{n,r}^k} (-1)^{\inv(\pi)} \prod _{i=1}^n m^r_{i,\pi_i}.\]
Now, for $\pi \in \MP_{n,r}^k$, let $T_{\pi} = \prod _{i=1}^n m^r_{i, \pi_i}$ be the term occuring in the permanent (respectively, determinant) expansion corresponding to $\pi$.
For $\pi \in \MP_{n,r}^k$, we have $m^r_{i,j}=t$ if $i < j$ and $m^r_{i,j}=1$ if $i > j$. Hence, $T_{\pi} = t^{\exc(\pi) }$. Thus, 
\[ \Perm(M_{n,r})= \displaystyle  \sum _{\pi \in \MP_{n,r}^k}  \prod _{i=1}^n t^{\exc(\pi)} = \MPE_{n,r}^k(t),\]
and 
\[ \det(M_{n,r}) =  \displaystyle  \sum _{\pi \in \MP_{n,r}^k} (-1)^{\inv(\pi)} \prod _{i=1}^n t^{\exc(\pi)} = \SMPE_{n,r}^k(t) .\] 
This completes the proof. 
\end{proof}

Recall that the \textit{Kronecker product} of two matrices $A=(a_{ij})_{m\times n}$ and $B=(b_{ij})_{p\times q}$, denoted by $A\otimes B$, is defined to be the $mp\times nq$ block matrix $[a_{ij}B]$. If $A$ and $B$ are $n \times n$ and $p\times p$ matrices, respectively, then $\det(A\otimes B)=(\det A)^p(\det B)^n$. We are now in a position to prove the first of our main results.

\begin{proof}
[\textbf{Proof of Theorem \ref{thm:signed excedance-modk}}]
We divide the proof into the following two cases:\\
\textbf{Case 1.} Let $n$ be not divisible by $k$. By Lemma \ref{lem:cruicial}, we have $\mpa_{n,r}^k=0$ for $r \geq 2$. Hence, $\SMPE_n^k(t)= \SMPE_{n,1}^k(t)$. Now, consider the  $n \times n$ matrix $M_{n,1}= (m^1_{i,j})$ defined in Proposition \ref{prop:mn1}, for which we have $\SMPE_{n,1}^k(t)=\det(M_{n,1})$. So, it is enough to compute the determinant of $M_{n,1}.$ By performing suitable row and column operations, we transform $M_{n,1}$ into a block diagonal matrix and compute its determinant as follows: 

Let $n=mk+j$, where $1 \leq j \leq k-1.$ We can write any positive integer $q \leq n$ as $q=ak+b$, where $0 \leq a \leq m$ and $ 0 \leq b \leq k-1$. If $0 < b \leq j$, we apply the operations $C_{ak+b}\rightarrow C_{a+1+(b-1)(m+1)}$ and $R_{ak+b}\rightarrow R_{a+1+(b-1)(m+1)}$; if $j < b \leq k-1$, we apply the operations $C_{ak+b}\rightarrow C_{j(m+1)+a+1+(b-j-1)m}$ and $R_{ak+b}\rightarrow R_{j(m+1)+a+1+(b-j-1)m}$; and if $b=0$, we apply the operations $C_{ak}\rightarrow C_{n-(m-a)}$ and $R_{ak}\rightarrow R_{n-(m-a)}$. This gives

\begin{align*}
\det(M_{n,1})=\det\left(
 	\begin{array}{cc}
 	A_{m+1}\otimes I_j & \bf{0}\\
     \bf{0} &  {A_{m}}\otimes I_{k-j}  
 	\end{array}
 	\right),~~\text{where}~~
  A_k= \left(
 	\begin{array}{cccc}
 	1 & t &  \cdots & t \\
 	1 & 1 &  \ddots & \vdots \\
 	\vdots & \ddots & \ddots & t \\
 	1 & \cdots & 1 & 1 \\
 	\end{array}
 	\right)_{k\times k}.
\end{align*}
Again, by applying the operations $C_i^\prime=C_i-C_{i+1}$ for $1\leq i\leq k-1$ on $A_k$, we have 
\begin{align}\label{det:A_k}
 \det(A_k)=\det \left(
 	\begin{array}{ccccc}
 	1-t & 0 &  \cdots & 0 & t \\
 	0 & 1-t &  \ddots & \vdots & t \\
 	\vdots & \ddots & \ddots & 0 & \vdots\\
        0 & \cdots & 0 & 1-t & t \\
 	0 & \cdots & 0 &  0 & 1 \\
 	\end{array}
 	\right)=(1-t)^{k-1}.   
\end{align}
Hence,
\begin{align*} 
\det(M_{n,1}) & =(\det(A_{m+1}))^j (\det(A_{m}))^{k-j}=(1-t)^{mj}(1-t)^{(m-1)(k-j)}=(1-t)^{n-k}.
\end{align*}

\noindent \textbf{Case 2.} Let $n$ be divisible by $k$, say $n=mk$. For $r=1$, the proof will be similar to Case 1 and we have $\det(M_{n,1})= (1-t)^{n-k}$. For $2 \leq r \leq k$, consider the $n \times n$ matrix $M_{n, r}= (m^r_{i, j})$ defined in Proposition \ref{prop:gen-mnr}. Again, by Proposition \ref{prop:gen-mnr}, we have $\SMPE_{n,r}^k(t)=\det(M_{n,r})$. So, it is sufficient to compute the determinant of $M_{n,r}$.

Let $n=mk$. We can write any positive integer $q \leq n$ as $q=ak+b$, where $0 \leq a \leq m$ and $ 0 \leq b \leq k-1$. If $b \neq 0$, we apply the operations $C_{ak+b}\rightarrow C_{a+1+m(b-1)}$; $R_{ak+b}\rightarrow R_{a+1+m(b-1)}$ and if $b=0$, we apply the operations $C_{ak}\rightarrow C_{a+m(k-1)}$; $R_{ak}\rightarrow R_{a+m(k-1)}$. This gives
\begin{align*}
\det(M_{n,r})=\det\left(
 	\begin{array}{cc}
 	\bf{0} & B_m \otimes I_{k-r+1}  \\
         A_m \otimes I_{r-1} & \bf{0} 
 	\end{array} 
 	\right),
  \end{align*}
where $A_m$ and $B_m$ are $m \times m$ matrices with
\begin{align*}
   A_m= \left(
 	\begin{array}{cccc}
 	1 & t &  \cdots & t \\
 	1 & 1 &  \ddots & \vdots \\
 	\vdots & \ddots & \ddots & t \\
 	1 & \cdots & 1 & 1 \\
 	\end{array}
 	\right)\text{and}~~ 
  B_m= \left(
 	\begin{array}{cccc}
 	t & t &  \cdots & t \\
 	1 & t &  \ddots & \vdots \\
 	\vdots & \ddots & \ddots & t \\
 	1 & \cdots & 1 & t \\
 	\end{array}
 	\right).
\end{align*}
 From (\ref{det:A_k}), we have $\det(A_m)=(1-t)^{m-1}$. Again, by applying the operations $R_i^\prime=R_i-R_{i+1}$ for $1\leq i\leq m-1$ on $B_m$, we have
$$\det(B_m)=\det \left(
 	\begin{array}{ccccc}
 	t-1 & 0 &   \cdots & 0 & 0 \\
 	0 & t-1 &  \ddots & \vdots & 0 \\
 	\vdots & \ddots & \ddots & 0 & \vdots\\
        0 & \cdots & 0 & t-1 & 0 \\
 	1 & \cdots & 1 & 1 & t \\
 	\end{array}
 	\right)=(-1)^{m-1}t(1-t)^{m-1}.$$ 
  Thus, 
  \begin{align*}
  \det (M_{n,r}) & =(-1)^{m^2(r-1)(k-r+1)}(\det(A_m))^{r-1}(\det(B_m))^{k-r+1} \\
  & =(-1)^{m(r-1)(k-r+1)+(m-1)(k-r+1)}t^{k-r+1} (1-t)^{(r-1)(m-1)+(m-1)(k-r+1)} \\
 & =(-1)^{(mr-1)(k-r+1)} t^{k-r+1} (1-t)^{n-k}.
  \end{align*}
  This completes the proof.
\end{proof}

\begin{remark}
Note that by putting $k=1$ in Theorem \ref{thm:signed excedance-modk}, we get Theorem \ref{thm:Mant_signed_exc} as an immediate corollary. 
\end{remark}

We get the following result as a by-product of the method of the proof of Theorem \ref{thm:signed excedance-modk}. 

\begin{theorem}
\label{thm:unsigned excedance-modk}
Let $n,k$ be positive integers. If $n$ is not divisible by $k$, say $n=mk+j$ with $1 \leq j \leq k-1$, then 
\[\MPE_n^k(t)= \MPE_{n,1}^k(t)= A_{m+1}(t)^{j} A_{m}(t)^{k-j}.  \] 
If $n$ is divisible by $k$, say $n=mk$, then 
\begin{eqnarray*}
\MPE_{n,r}^k(t)  & = & \begin{cases}
A_m(t)^k & \text {if $r=1$}, \\
t^{k+1-r} A_m(t)^k & \text {if $2 \leq r \leq k$}.
 \end{cases}
\end{eqnarray*} 
\end{theorem}

\begin{proof}
First, let us consider the case when $n$ is not divisible by $k$. 
By Lemma \ref{lem:cruicial}, it follows that $\MPE_n^k(t)= \MPE_{n,1}^k(t)$. Now, consider the  $n \times n$ matrix $M_{n,1}= (m^1_{i,j})$ defined in Proposition \ref{prop:mn1}, for which we have $\MPE_{n,1}^k(t)=\Perm(M_{n,1})$. So, it is enough to compute the permanent of $M_{n,1}.$  Following the proof of Theorem \ref{thm:signed excedance-modk}, we have 
\[ \Perm(M_{n,1})=(\Perm(A_{m+1}))^j (\Perm(A_{m}))^{k-j}, \]
where $A_m$ is the same matrix
as defined in the proof of Theorem \ref{thm:signed excedance-modk}. Since $\Perm(A_m)$ is same as the Eulerian polynomial $A_m(t)$, we have 
\[ \MPE_n^k(t)= \MPE_{n,1}^k(t) =  (\Perm(A_{m+1}))^j (\Perm(A_{m}))^{k-j} 
=A_{m+1}(t)^j A_m(t)^{k-j} .\]
 
Next, we consider the case when $n$ is divisible by $k$, say $n=mk$.
For $r=1$, it is easy to see that 
\[ \MPE_{n,1}^k(t) = A_m(t)^k .\] 
For $2 \leq r \leq k$, from the proof of Theorem \ref{thm:signed excedance-modk} Case 2, we have 
\begin{align*}
  \MPE_{n,r}^k(t)  = \Perm (M_{n,r}) &= (\Perm(A_{m}))^{r-1} (\Perm(B_{m}))^{k-r+1} \\  
& = A_m(t)^{r-1} (tA_m(t))^{k-r+1} \\
& = t^{k-r+1} A_m(t)^k.
\end{align*}
This completes the proof. 
\end{proof}


\section{Mod-k-alternating derangements} 
\label{sec:mod-k-alt-derang}

 In this section, we study the sign-balance for excedances over mod-k-alternating derangements. 
The following lemma will be useful to get an explicit expression for the polynomial $\SMPDE_{n,r}^k(t)$. 

\begin{lemma}[{\cite[Theorem 2.1]{Mikkawy-det-tridiagonal}}]
\label{lem:det-tridisgonal}
Let
    \begin{align*}
      f_n=\det \left(
 	\begin{array}{cccccc}
 	c_1 & a_1 & 0 & \cdots & \cdots & 0 \\
 	b_2 & c_2 & a_2 & \ddots &   & \vdots \\
 	0 & b_3 & c_3 & \ddots &  0 & \vdots \\
 	\vdots & \ddots & \ddots & \ddots  & \ddots & 0 \\
        \vdots & & 0 & \ddots & \ddots & a_{n-1} \\
        0 & \cdots & \cdots & 0 & b_n & c_n
 	\end{array}
 	\right).  
    \end{align*}
Then $f_n$ satisfies the following recurrence relation:
$$f_n=c_nf_{n-1}-b_na_{n-1}f_{n-2},$$
where the initial values for $f_n$ are $f_0=1$ and $f_{-1}=0$.
\end{lemma}

\begin{proof}[\textbf{Proof of Theorem \ref{thm:signed_Exc_derange_pap}}]
Note that any mod-k-alternating permutation starting with the remainder $r\pmod k$ is a derangement when $r \geq 2$. So, we only need to consider the case when $r=1$.
Let $D_n= (d_{i,j})$ be an $n \times n$ matrix with
\begin{eqnarray*}
	d_{i,j} & = & \begin{cases}
	0 & \text {if $|i-j|$ is not divisible by $k$ or $i=j$}, \\
	 1 & \text {if $|i-j|$ is divisible by $k$ and $i > j$},\\
	  t & \text {if $|i-j|$ is divisible by $k$ and $i < j$}. 
	 \end{cases}
\end{eqnarray*}   		

Note that
\begin{eqnarray*}
\det(D_n) & = & \displaystyle  \sum _{\pi \in \SSS_n} (-1)^{\inv(\pi)} \prod _{i=1}^n d_{i,\pi_i} \\
  & = & \displaystyle  \sum _{\pi \in \MPD_{n,1}^k } (-1)^{\inv(\pi)} \prod _{i=1}^n d_{i,\pi_i} + \displaystyle  \sum _{\pi \in (\SSS_n - \MPD_{n,1}^k) } (-1)^{\inv(\pi)} \prod _{i=1}^n d_{i,\pi_i}.  
  \end{eqnarray*} 
For a permutation $\pi \not \in \MPD_{n,1}^k$, we must have $|\pi_i-i|$ is not  divisible by $k$ for some $1 \leq i \leq n$, or  $\pi_i =i$ for some $1 \leq i \leq n$. Hence, if  $\pi \not \in \MPD_{n,1}^k$,  we must have $d_{i,\pi_i}=0$. Thus, we have 
\[ \det(D_n) = \sum _{\pi \in \MPD_{n,1}^k } (-1)^{\inv(\pi)} \prod _{i=1}^n d_{i,\pi_i}. \] 
Now, let $T_{\pi} = \prod _{i=1}^n d_{i, \pi_i}$ be the term occurring in the determinant expansion corresponding to $\pi$. From the definition of $d_{i, \pi_i}$, we have $T_{\pi} = t^{\exc(\pi) }$. Hence, $\SMPDE_{n,1}^k(t)=\det(D_n).$ Now, to find the determinant of $D_n$, we transform $D_n$ into a block diagonal matrix by applying suitable row and column operations.

Let $n=mk+j$, where $1 \leq j \leq k-1.$ We can write any positive integer $q \leq n$ as $q=ak+b$ where $0 \leq a \leq m$ and $ 0 \leq b \leq k-1$. If $0 < b \leq j$, we apply the operations $C_{ak+b}\rightarrow C_{a+1+(b-1)(m+1)}$ and $R_{ak+b}\rightarrow R_{a+1+(b-1)(m+1)}$; if $j < b \leq k-1$, we apply the operations $C_{ak+b}\rightarrow C_{j(m+1)+(b-j-1)m+a+1}$ and $R_{ak+b}\rightarrow R_{j(m+1)+(b-j-1)m+a+1}$; and if $b=0$, we apply the operations $C_{ak}\rightarrow C_{n-(m-a)}$ and $R_{ak}\rightarrow R_{n-(m-a)}$. This gives

 \begin{align*}
\det(D_n)= \det\left(
 	\begin{array}{cc}
 	A_{m+1}\otimes I_j & \bf{0}\\
     \bf{0} &  {A_{m}}\otimes I_{k-j}  
 	\end{array}
 	\right),~~\text{where}~~
  A_k= \left(
 	\begin{array}{cccc}
 	0 & t &  \cdots & t \\
 	1 & 0 &  \ddots & \vdots \\
 	\vdots & \ddots & \ddots & t \\
 	1 & \cdots & 1 & 0 \\
 	\end{array}
 	\right)_{k\times k}.
\end{align*}
Again, by applying the operations $C_i^\prime=C_i-C_{i+1}$ and $R_i^\prime=R_i-R_{i+1}$ for $1\leq i\leq k-1$ on $A_k$, we have 
\begin{align*}
 \det(A_k)=\det \left(
 	\begin{array}{cccccc}
 	-1-t & t & 0 &  \cdots & 0 & 0 \\
 	1 & -1-t & t & \ddots & \vdots & \vdots \\
 	0 & 1 & -1-t &  \ddots & 0 & 0 \\
 	\vdots & \ddots & \ddots & \ddots & t & 0\\
        0 & \cdots & 0 & 1 & -1-t & t \\
 	0 & \cdots & 0 & 0 & 1 & 0 \\
 	\end{array}
 	\right)  .
\end{align*}
Now, expanding the determinant with respect to the last row, we have 
\begin{align*}
 \det(A_k)=(-1)^{2k-1}\det \left(
 	\begin{array}{cccccc}
 	-1-t & t & 0 &  \cdots & 0 & 0 \\
 	1 & -1-t & t & \ddots & \vdots & \vdots \\
 	0 & 1 & -1-t &  \ddots & 0 & 0 \\
 	\vdots & \ddots & \ddots & \ddots & t & 0\\
        0 & \cdots & 0 & 1 & -1-t & 0 \\
 	0 & \cdots & 0 & 0 & 1 & t \\
 	\end{array}
 	\right)_{(k-1)\times (k-1)}.   
\end{align*}
Again, expanding with respect to the last column, we have
\begin{align*}
 \det(A_k)=(-1)^{2k-1}t \det \left(
 	\begin{array}{ccccc}
 	-1-t & t & 0 &  \cdots & 0 \\
 	1 & -1-t & t & \ddots & \vdots \\
 	0 & 1 & -1-t &  \ddots & 0\\
 	\vdots & \ddots & \ddots & \ddots & t\\
        0 & \cdots & 0 & 1 & -1-t  \\
 	\end{array}
 	\right)_{(k-2)\times (k-2)}      
\end{align*}
Now, by Lemma \ref{lem:det-tridisgonal}, we have $ \det(A_k)=(-1)^{2k-1}t(-1)^{k-4}(1+t+t^2+\hdots+t^{k-2})=(-1)^{k-1}t[k-1]_t$. 
Thus, 
\begin{align*}
\det(D_n) & =(\det(A_{m+1}))^j  (\det(A_m))^{k-j} \\
& = (-1)^{mj} t^j \big([m]_t \big)^j
(-1)^{(m-1)(k-j)} t^{k-j} \big([m-1]_t\big)^{k-j}\\
& = (-1)^{n} (-t)^k \big([m]_t \big)^ j \big([m-1]_t\big)^{k-j}. 
\end{align*}
This completes the proof for $r=1$ and hence the proof of the theorem.
\end{proof}


\section{Gamma-positivity of even and odd mod-k-alternating permutations}
\label{sec:gamma-positivity}

Let $f(t)  = \sum_{i=0}^n a_i t^i \in \QQ[t]$ be a  univariate polynomial of degree $n$, where $a_i \in \QQ$ with $a_n \not= 0$.  Let $r$ be the least non-negative integer such that $a_r \not= 0$.  The polynomial $f(t)$ is said to be {\it palindromic} if $a_{r+i} = a_{n-i}$ for $0 \leq i \leq \floor{(n-r)/2}$.  Define the {\it center of symmetry} of $f(t)$ to be $(n+r)/2$.  Note that for a palindromic polynomial $f(t)$, its center of symmetry could be half integral.

Let $\PP_{(n+r)/2, r}(t)$ denote the set of palindromic 
univariate polynomials $f(t)$ with minimum nonzero exponent of $t$ being at least $r$ and having center of symmetry $(n+r)/2$.  Let 
$\Gamma = \{ t^{r+i}(1+t)^{n-r-2i}: 0 \leq i \leq \nmrhalf \}$.
It is easy to see that if $f(t) \in \PP_{(n+r)/2, r}(t)$, then we can write $f(t) = \sum_{i=0}^{\floor{(n-r)/2}} \gamma_{n,i} t^{r+i} (1+t)^{n-r-2i}$. The polynomial $f(t)$ is said to be {\it gamma-positive} if $\gamma_{n,i} \geq 0$ for all $i$.

The study of gamma-positivity is an interesting and important topic in enumerative combinatorics. It appears widely in finite geometries, combinatorics and number theory.  Foata and Sch{\"u}tzenberger \cite{foata-schu} showed gamma-positivity of the Eulerian polynomials $A_n(t)$ that enumerate descents in the symmetric group $\mathfrak{S}_n.$ Subsequently, Foata and Strehl \cite{foata-strehl} 
used a group action based proof which has been rediscovered and termed as ``valley hopping" by Shapiro, Woan, and Getu. More recent interest in gamma-positivity was sparked by Gal \cite{gal} in 2005 when he showed that some questions in topology could be resolved by demonstrating the gamma-positivity of their combinatorial invariants. For more details on the gamma-positivity, one can look at \cite{athanasiadis,han,ma-ma-yeh-descent,Ma-ma-yeh} and the references therein. 

In this section, we are interested in the gamma-positivity of the excedance enumerating polynomials over mod-k-alternating permutations. First, we give some preliminary results which will be used to prove the main result. 

\begin{lemma}[{\cite[Lemma 1]{Dey-Siva-annals-com}}]
\label{lem:gamma-pos-under-prod}
Let $f_1$ and $f_2$ be two gamma-positive polynomials with centers of symmetry $a_1$ and $a_2$, respectively. Then, their product $f_1f_2$ is a gamma-positive polynomial with center of symmetry $a_1+a_2$. 
\end{lemma}

\begin{lemma}[{\cite[Theorem 4]{Dey-Siva-annals-com}}]
\label{lem:gamma-pos-excedance}
For odd positive integers $n\geq 5$, the polynomials 
$\sum_{\pi \in \AAA_n} t^{\exc(\pi)}$ and $\sum_{\pi \in (\SSS_n -\AAA_n)} t^{\exc(\pi)}$ are gamma-positive and both of them have center of symmetry $(n-1)/2.$
\end{lemma}

\begin{lemma}
\label{lem:fgtofsqrgsqr}
Let $f,g$ be palindromic polynomials with same center of symmetry such that $f+g$ and $f-g$ are gamma-positive with same center of symmetry $a$. Then, for all positive integers $r$, the polynomials $f^r+g^r$ and $f^r-g^r$ are gamma-positive with center of symmetry $ra$. 
\end{lemma}

\begin{proof}
We will prove this by induction on $r$. When $r=1$, the statement is true as $f+g$ and $f-g$ are gamma-positive with the same center of symmetry. Let us assume that for all positive integers $r \leq k$, the polynomials $f^r+g^r$ and $f^r-g^r$ are gamma-positive with center of symmetry $ra$. Now, consider $r=k+1$. It is easy to check that 
\[f^{k+1}+g^{k+1}= \frac{1}{2} \big[ (f+g)(f^{k}+g^{k})+ (f-g)(f^{k}-g^{k}) \big] ,\]
and 
\[f^{k+1}-g^{k+1}= \frac{1}{2} \big[ (f+g)(f^{k}-g^{k})+ (f-g)(f^{k}+g^{k}) \big] .\]
By Lemma \ref{lem:gamma-pos-under-prod}, it follows that $(f+g)(f^k+g^k)$ is gamma-positive with center of symmetry $ka+a=(k+1)a$ and the same holds for $(f-g)(f^k-g^k).$ Hence, $f^r+g^r$ and $f^r-g^r$ are gamma-positive with center of symmetry $ra$ for $r=k+1$. This completes the proof. 
\end{proof}

\begin{remark}
\label{rem:gamma-pos}
Using Theorem \ref{thm:unsigned excedance-modk} and Lemma 
\ref{lem:gamma-pos-under-prod}, it is immediate that the polynomials $\MPE_{n,r}^k(t)$ are gamma-positive. 
\end{remark}

\begin{remark}
\label{rem:not-gamma-pos}
From Theorem \ref{thm:signed excedance-modk} and Theorem \ref{thm:unsigned excedance-modk}, we can see that the polynomials 
$\MPE_{n,r}^{k,+}(t)$ and $\MPE_{n,r}^{k,-}(t)$ are not  palindromic unless $n$ is divisible by $k$. Therefore, there is no question of gamma-positivity unless $n$ is divisible by $k$. 
\end{remark}

Now, by using the above results, we give the proof of Theorem \ref{thm:gamma-pos-pap-even-odd-n-2mod4}.
We use the notation $\MPE_{n,r}^{k,\pm}(t)$ to refer to both
the polynomials 
$\MPE_{n,r}^{k,+}(t)$ and $\MPE_{n,r}^{k,-}(t)$.

\begin{proof}[\textbf{Proof of Theorem \ref{thm:gamma-pos-pap-even-odd-n-2mod4}}]
Let $n \equiv k \pmod {2k}$ with $n \geq 5k$. Suppose $n=2km+k$,  where $m$ is a positive integer with $m\geq 2$. Then, we have

\[\MPE_{2mk+k, r}^{k,\pm}(t) = \frac{1}{2} \left[ \sum_{\pi \in \MPE_{2mk+k,r}^{k}  }  t^{\exc(\pi)} \pm \sum_{\pi \in \MPE_{2mk+k,r}^{k}  }  \sgn(\pi) t^{\exc(\pi)}\right] .\]
By Theorem \ref{thm:unsigned excedance-modk} and Theorem \ref{thm:signed excedance-modk}, it follows that

\[ \MPE_{2mk+k,1}^{k,\pm}(t) = \frac{1}{2} \left[ A_{2m+1}(t)^k \pm  (1-t)^{2mk}  \right] ~ \text{for} ~ r=1, \]
and for $r \geq 2$
\[ \MPE_{2mk+k,r}^{k,\pm}(t) = \frac{1}{2} t^{k+1-r} \left[ A_{2m+1}(t)^k \pm 
(-1)^{(m^2r-1)(k+1-r)}   (1-t)^{2mk}  \right]. \] 
Now, by Lemma \ref{lem:gamma-pos-excedance}, the polynomials $A_{2m+1}(t)+(1-t)^{2m}$ and $A_{2m+1}(t)-(1-t)^{2m}$ are gamma-positive with centers of symmetry 
$m$. Therefore, by Lemma \ref{lem:fgtofsqrgsqr} it follows that the polynomials $\MPE_{2mk+k,1}^{\pm,k}(t)$ are gamma-positive with the center of symmetry $km$. Similarly, for $2 \leq r \leq k$, the polynomials $\MPE_{n,r}^{k,+}(t)$ and $\MPE_{n,r}^{k,-}(t)$  are gamma-positive with center of symmetry  $(k+1-r+2km)/2.$ 
\end{proof} 

	 
\section*{Declaration of competing interest}	 
	
The authors declare that they have no conflict of interest to this work.
	
\section*{Data availability}	
	
No data was used for the research described in the article.

\section*{Acknowledgements}
The authors would like to thank Prof. Sivaramakrishnan Sivasubramanian for 
valuable comments and suggestions during the preparation of the manuscript. The first author acknowledges a NBHM Post-Doctoral Fellowship (File No. 0204/10(10)/2023/R{\&}D-II/2781) during the preparation of this work and profusely thanks National Board of Higher Mathematics, India for this funding. The first author also acknowledges excellent working conditions in the Department of Mathematics, Indian Institute
of Science. The second author thanks Indian Institute of Technology Bombay, India for financial support through the Institute Post-Doctoral Fellowship.


\end{document}